\documentclass{article}[12pt]
\usepackage{url, graphicx, xcolor, hyperref, geometry}
\usepackage{latexsym,amsbsy}
\usepackage{lipsum}
\usepackage{amsfonts}
\usepackage{graphicx}
\usepackage{epstopdf}
\usepackage{algorithmic}
\usepackage{tikz}
\usetikzlibrary{matrix}
\usetikzlibrary{shapes,patterns,arrows,snakes,decorations.shapes}
\usetikzlibrary{positioning,fit,calc}
\usetikzlibrary{plotmarks}
\usepackage{amssymb,amsmath,amsthm}
\usepackage{smartdiagram}
\usepackage{mathabx}
\usepackage{float}
\usepackage{enumerate}

\newtheorem{theorem}{Theorem}[section]
\newtheorem{lemma}{Lemma}[section]

\newtheorem{remark}{Remark}[section]

\newtheorem{definition}{Definition}[section]

\usepackage[font=small,labelfont=bf, labelsep=space]{caption}
\hypersetup{
    colorlinks,
    linkcolor={blue},
    citecolor={blue},
    urlcolor={blue}
}

\begin{document}
 
\title{Novel  mass-based  multigrid relaxation schemes for the Stokes equations}

\author{Yunhui He\thanks{Department of Computer Science, The University of British Columbia, Vancouver, BC, V6T 1Z4, Canada,   \tt{yunhui.he@ubc.ca}.}  }

\maketitle

\begin{abstract}
 In this work, we propose  three novel block-structured multigrid relaxation schemes based on distributive relaxation,  Braess-Sarazin relaxation, and Uzawa
relaxation, for solving the Stokes equations discretized by the mark-and-cell scheme. In our earlier work \cite{he2018local}, we  discussed these three types of relaxation schemes, where the weighted Jacobi iteration  is used for inventing the Laplacian involved in the Stokes equations. In \cite{he2018local}, we show that  the optimal smoothing factor is  $\frac{3}{5}$  for  distributive weighted-Jacobi relaxation and  inexact Braess-Sarazin relaxation,  and  is $\sqrt{\frac{3}{5}}$ for $\sigma$-Uzawa relaxation.  Here, we propose mass-based approximation inside of these three relaxations, where mass matrix $Q$ obtained from  bilinear finite element method is  directly used to approximate to the inverse of  scalar Laplacian operator instead of using Jacobi iteration. Using local Fourier analysis, we  theoretically derive the optimal smoothing factors for the resulting three relaxation schemes.  Specifically,  mass-based  distributive relaxation,  mass-based  Braess-Sarazin  relaxation, and mass-based $\sigma$-Uzawa relaxation have  optimal smoothing factor $\frac{1}{3}$, $\frac{1}{3}$ and  $\sqrt{\frac{1}{3}}$, respectively.  Note that the mass-based relaxation schemes do  not cost more than the original ones using Jacobi iteration. Another superiority  is that there is no need to compute the  inverse of a matrix. These new relaxation schemes are appealing.   
\end{abstract}

 
 \vspace{2mm}

\section{Introduction}
\label{sec:intro}
 
 Multigrid methods \cite{MR1156079} are popular for numerical solution of  large classes of problems, such as the Stokes and  Navier-Stokes equations, see \cite{wittum1989multi, oosterlee2006multigrid,MR833993,MR3217219}, due to their   capability of solving problems with $N$ unknowns using $O(N)$ work and storage, leading to a substantial improvement in computational efficiency over direct methods.  In this work, we are interested in multigrid methods for the Stokes equations discretized by a staggered finite-difference method. To design fast multigrid methods, the choices of multigrid components, such as relaxation schemes and  grid-transfer operators,  are very crucial.    Here, we consider three block-structured relaxation schemes, Braess-Sarazin relaxation \cite{braess1997efficient},  Uzawa relaxation \cite{MR833993}, and distributive relaxation \cite{bacuta2011new}.  Braess-Sarazin relaxation  has been used for the Stokes equations \cite{braess1997efficient,he2018local, he2019local,voronin2021low},  poroelasticity equations \cite{adler2021monolithic}, magnetohydrodynamic equations \cite{adler2016monolithic}.  Uzawa-type relaxation has been applied to different problems, see  \cite{MR833993} for the Stokes equations, \cite{luo2017uzawa} for poroelasticity equations.  Analysis of inexact Uzawa for saddle-point problem  can be found in \cite{MR1451114,MR1302679,MR2001083}.   The papers \cite{bacuta2011new,oosterlee2006multigrid,MR3071182} investigate  distributive relaxation  for the Stokes equations,  and \cite{chen2015multigrid} is for Oseen problems.


In \cite{he2018local}, we note that solving the Poisson equation arising as  a subproblem in the Stokes equation plays an important  role in determining the convergence speed in the multigrid methods.  Motivated by our recent work  \cite{CH2021addVanka},  where mass matrix $Q$ obtained from bilinear finite elements in two dimensions has been shown to be a   good approximation in some sense to the inverse of scalar Laplacian discretized by five-point finite difference method, here, we propose three novel block-structured relaxation schemes: mass-based  Braess-Sarazin relaxation ($Q$-BSR),  mass-based $\sigma$-Uzawa relaxation ($Q$-$\sigma$-Uzawa), and mass-based distributive relaxation ($Q$-DR), for solving the Stokes equations discretized by staggered finite-difference  scheme.  We present a theoretical analysis of the optimal smoothing factor of local Fourier analysis (LFA) \cite{MR1807961} for these relaxation schemes.  Our theoretical results show that  the optimal smoothing factors for the mass-based  distributive relaxation and mass-based Braess-Sarazin  relaxation are both  $\frac{1}{3}$, and it is   $\sqrt{\frac{1}{3}}$ for the mass-based $\sigma$-Uzawa relaxation.  To consider a potential parallel implementation for $Q$-BSR, we explore  inexact version of $Q$-BSR, called $Q$-IBSR,  where we apply one sweep of  weighted-Jacobi iteration to the Schur complement system. We find that $Q$-IBSR preserves the optimal smoothing factor $\frac{1}{3}$. Note that the mass matrix $Q$ is sparse and  computing the inverse of a matrix  is avoided in these relaxation schemes.  Our optimal smoothing factors  here are smaller than the corresponding ones ($\frac{3}{5}$ and $\sqrt{\frac{3}{5}}$) in \cite{he2018local}, where the weighted-Jacobi iteration is employed in solving the system with Laplacian coefficient matrix.  Regarding the computational cost of these mass-based relaxation schemes, we note that the cost of 2 sweeps of $Q$-$\sigma$-Uzawa is slightly more than one sweep of $Q$-IBSR, and the cost of one sweep of  $Q$-DR is a slightly more than the cost of one sweep of $Q$-IBSR. Consequently,  $Q$-IBSR is the most efficient.

An outline of the paper is as follows.   In Section \ref{sec:discretization}, we introduce the  staggered finite-difference discretization  for the Stokes equations and propose three mass-based block-structured multigrid relaxation schemes.  In Section \ref{sec:smoothing-analysis}, we present the analytical optimal smoothing factor of LFA for these three mass-based relaxation schemes.   Numerical results are presented in Section \ref{sec:Numer} to validate our theoretical results. Finally, we draw some conclusions in Section \ref{sec:concl}.
\section{Discretization and relaxation schemes}\label{sec:discretization}

In this work, we  consider  the following Stokes equations in two dimensions
\begin{eqnarray}
  -\triangle\boldsymbol{u}+\nabla p&=&\boldsymbol{f},\label{Stokes1}\\
  \nabla\cdot \boldsymbol{u}&=&0,\label{Stokes2}
\end{eqnarray}
where $\boldsymbol{u}=\begin{pmatrix} u \\ v \end{pmatrix}$ is the velocity vector, and  $p$ is the scalar pressure of a viscous fluid.

We apply the standard staggered finite-difference discretization to \eqref{Stokes1} and \eqref{Stokes2}, known as the mark-and-cell (MAC) scheme, see \cite{MR1807961},  where the unknowns $u, v, p$, are placed in different locations, see  Figure~\ref{fig:MAC-stokes}. 
The velocity $u$ is located on the middle points of vertical edges ($\Box$), the velocity $v$ is located  on middle points of horizontal edges ($\lozenge$), and the pressure  $p$ is located in the center of each cell ($\bigstar$).  

 \begin{figure}[htp]
\centering
\includegraphics[width=0.5\textwidth]{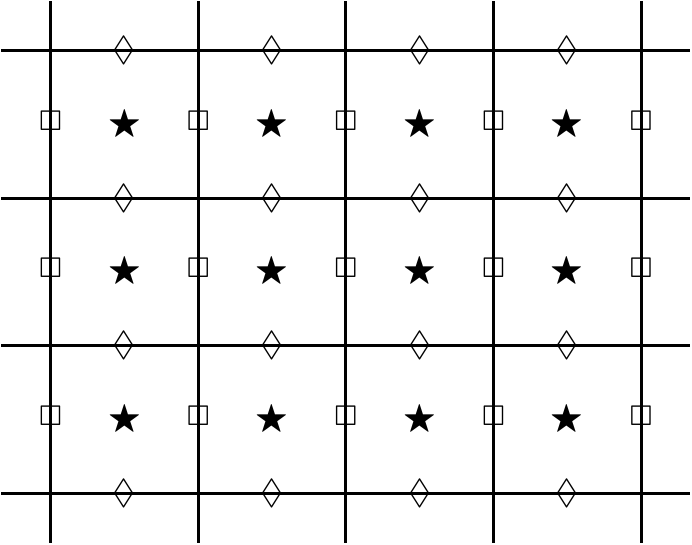}
 \caption{The staggered location of unknowns: $\Box-u,\,\, \lozenge-v,  \,\, \bigstar-p$}\label{fig:MAC-stokes}
\end{figure}

In this paper, we deal with  uniform meshes with $h_x=h_y=h$.   The discrete momentum equations is given as  \cite{MR1807961}
\begin{align*} 
  -\triangle_{h}u_h+(\partial_{x})_{h/2}\,p_{h}&=f_{1,h}, \\
 -\triangle_{h}v_h+(\partial_{y})_{h/2}\,p_{h}  &=f_{2,h}.
\end{align*}
The Laplacian $-\triangle_{h}$ is discretized by the standard five-point scheme as follows
\begin{align*} 
  (\partial_{x})_{h/2}\,p_{h}(x,y) &= \frac{1}{h}\bigg(p_{h}\big(x+h/2,y\big)-p_{h}\big(x-h/2,y\big)\bigg), \\
  (\partial_{y})_{h/2}\,p_{h}(x,y) &= \frac{1}{h}\bigg(p_{h}\big(x,y+h/2\big)-p_{h}\big(x,y-h/2\big)\bigg).
\end{align*}
The discrete conservation of mass equation is given by
\begin{equation*}
 (\partial_{x})_{h/2}\,u_{h}(x,y)+(\partial_{y})_{h/2}\,v_{h}(x,y)=0.
\end{equation*}
The staggered discretization of the Stokes equations then  reads \cite{MR1807961}
 \begin{equation}\label{eq:Lh-operator}
   \mathcal{L}_h   =\begin{pmatrix}
      -\triangle_{h} & 0 &  (\partial_{x})_{h/2}\\
      0 & -\triangle_{h} & (\partial_{y})_{h/2} \\
      -(\partial_{x})_{h/2}  & -(\partial_{y})_{h/2} & 0
    \end{pmatrix},
  \end{equation}
where  the computational stencils are given by
\begin{equation*}
   -\triangle_{h} =\frac{1}{h^2}\begin{bmatrix}
       & -1 &  \\
      -1 & 4 & -1 \\
        & -1 &
    \end{bmatrix},\quad
     (\partial_x)_{h} =\frac{1}{h}\begin{bmatrix}
       -1& 0 & 1 \\
    \end{bmatrix},\quad
     (\partial_{y})_{h} =\frac{1}{h}\begin{bmatrix}
       1 \\
      0 \\
      -1
    \end{bmatrix}.
\end{equation*}

\subsection{Relaxation schemes}
The resulting linear system of (\ref{Stokes1}) and (\ref{Stokes2})  is
\begin{equation}\label{saddle-structure}
       \mathcal{L}_{h}  x=\begin{pmatrix}
      A & B^{T}\\
     B & 0\\
    \end{pmatrix}
        \begin{pmatrix} \boldsymbol{u}_{h} \\ p_{h}\end{pmatrix}
  =\begin{pmatrix} \boldsymbol{f}_h \\ 0 \end{pmatrix}=b,
 \end{equation}
where $A$ is  the discretized vector Laplacian, $B$ is the negative of the discrete divergence operator, and $\boldsymbol{u}_{h}=\begin{pmatrix} u_{h} \\ v_{h} \end{pmatrix}$.

To solve \eqref{saddle-structure}, we consider the following relaxation scheme: given a current approximation $x_k$,  we updates this approximation via
\begin{equation}\label{eq:general relaxation-form}
\hat{x}_{k} = x_k-\omega M^{-1}(b-  \mathcal{L}_{h}  x_k)=x_k-\omega \delta x,
\end{equation}
where $M$ is an approximation to  $\mathcal{L}_{h} $, and $\delta x =M^{-1}(b-  \mathcal{L}_{h} x_k)$.   Then, the error-propagation for the relaxation scheme \eqref{eq:general relaxation-form} is 
\begin{equation}\label{eq:Sh-general-form}
  \mathcal{S}_h =I-\omega  { M}^{-1}  \mathcal{L}_h.
\end{equation}
The spectral radius of  $\mathcal{S}_h$ often determines the convergence speed in the multigrid methods.  In this work, we will  employ local Fourier analysis to    minimize the spectral radius of  $\mathcal{S}_h$, which will be presented in Section \ref{sec:smoothing-analysis}.

Note that \cite{he2018local}   discussed three block-structured relaxation schemes, that is, distributive weighted-Jacobi relaxation, Braess-Sarazin relaxation, and Uzawa-type
relaxation within multigrid methods  for solving \eqref{saddle-structure}.  We first give a brief introduction to these three relaxations, then propose our improved relaxation schemes.  For comparison purposes, we borrow some notations  from \cite{he2018local}.  For simplicity, throughout the rest of this paper,   we drop the subscript
$h$, except when necessary for clarity.

{\bf Distributive relaxation}: We detail the idea of distributive relaxation \cite{MR558216} as follows. To relax the equation $\mathcal{L}x = b$, we  consider a transformed system $\mathcal{K}\hat{x}=\mathcal{L}P\hat{x}= b$, where $P\hat{x}=x$ and $P$ is chosen so that   $\mathcal{K}$
is suitable for decoupled relaxation.  For the Stokes equations,  $P $  is taken to be
 \begin{equation*}
  P =\begin{pmatrix}
      I_{h} & 0 &  (\partial_{x})_{h/2}\\
      0 & I_{h} & (\partial_{y})_{h/2} \\
      0  &0  & \triangle_{h}
    \end{pmatrix},
  \end{equation*}
whose  discrete matrix form  is
\begin{equation*}\label{P-Precondtion}
   P=  \begin{pmatrix}
      I & B^{T}\\
      0 & -A_p\\
    \end{pmatrix},
\end{equation*}
where $A_p$ is the  the standard 5-point stencil of the Laplacian operator discretized at cell centers.  Accordingly,  the distributed operator has the form
   \begin{equation}\label{DWJ-system}
      \mathcal{K}=\mathcal{L}P=\begin{pmatrix}
      -\triangle_{h} & 0 &  0\\
      0 & -\triangle_{h} & 0 \\
      -(\partial_{x})_{h/2}  &-(\partial_{y})_{h/2}  & -\triangle_{h}
    \end{pmatrix}.
  \end{equation}
Then, we can apply  block   relaxation to the transformed system.  For instance, distributive Gauss-Seidel relaxation \cite{MR558216,oosterlee2006multigrid} (DGS) is well-known to be highly efficient for the MAC discretization.   Distributive weighted-Jacobi (DWJ) relaxation has been discussed in \cite{he2018local},  where one needs to  solve the following system
\begin{equation}\label{DWJ-Precondtion}
   M_D\delta \hat{x}=  \begin{pmatrix}
      \alpha_D C & 0\\
      B & \alpha_D E \\
    \end{pmatrix}
    \begin{pmatrix} \delta \mathcal{\hat{U}} \\ \delta \hat{p}\end{pmatrix}
  =\begin{pmatrix} r_{\mathcal{U}} \\ r_{p}\end{pmatrix},
\end{equation}
where  $C= {\rm diag}(A)$, and $E={\rm diag}(A_p)$. Furthermore,
\begin{eqnarray*}
  \delta \mathcal{\hat{U}} &=&(\alpha_D  C)^{-1}r_{\mathcal{U}}, \\
 \delta \hat{p} &=& \big(\alpha_D  E \big)^{-1}(r_{p}-B\delta \mathcal{\hat{U}}).
\end{eqnarray*}
 Using $\delta x=\mathcal{P}\delta \hat{x}$ gives 
\begin{eqnarray*}
  \delta \mathcal{U} &=&\delta \mathcal{\hat{U}}+B^{T}\delta \hat{p},\\
 \delta p &=& -A_{p} \delta \hat{p}.
 \end{eqnarray*}
 The error propagation operator for the distributive relaxation scheme is  $I-\omega_DPM^{-1}_D\mathcal{L}$.

 {\bf  Braess-Sarazin relaxation}: In Braess-Sarazin relaxation, $M$ is given by
\begin{equation}\label{Precondtion}
   M_B=  \begin{pmatrix}
      \alpha C & B^{T}\\
     B & 0\\
    \end{pmatrix},
\end{equation}
where $C$ is an approximation of $A$, whose inverse is easy to apply, and $\alpha>0$ is a chosen relaxation parameter. 

Using $M=M_B$ in \eqref{eq:general relaxation-form}, one needs to compute the update $\delta x=(\delta \mathcal{U}, \delta p)$ in \eqref{eq:general relaxation-form} by  the following two stages
\begin{eqnarray}
  (BC^{-1}B^{T})\delta p&=&BC^{-1}r_{\mathcal{U}}-\alpha r_{p}, \label{solution-of-precondtion}\\
  \delta \mathcal{U}&=&\frac{1}{\alpha}C^{-1}(r_{\mathcal{U}}-B^{T}\delta p),\nonumber
\end{eqnarray}
where $(r_{\mathcal{U}},r_{p})=b-\mathcal{L}x_k$.

In \cite{he2018local}, we considered  $C={\rm diag}(A)$.  In practice,   an approximate solve  for \eqref{solution-of-precondtion} is sufficient \cite{MR1810326}, such as using a simple sweep of a Gauss-Seidel \cite{MR1049395} or weighted Jacobi iteration \cite{he2018local}.

{\bf $\sigma$-Uzawa relaxation} \cite{he2018local}:  Uzawa-type relaxation is  a popular family of algorithms for solving saddle-point systems \cite{MR1302679,MR833993}, which can be treated as a lower triangular approximation, where $M$ is taken to be
\begin{equation}\label{Uzawa-Precondtion}
   M_{U} =  \begin{pmatrix}
      \alpha C & 0\\
     B & -D\\
    \end{pmatrix},
\end{equation}
where $\alpha C$ is an approximation of $A$, for example, $C={\rm diag}(A)$ \cite{he2018local},  and $D$ is an approximation to the Schur complement $BC^{-1}B^{T}$.  One possible choice of  $D$ is  $D =\sigma^{-1}I$, see \cite{he2018local}. The resulting relaxation is referred to $\sigma$-Uzawa relaxation.

In $\sigma$-Uzawa relaxation,  $\delta x$ in \eqref{eq:general relaxation-form}  with $M=M_U$ defined in (\ref{Uzawa-Precondtion}) is equivalent to computing updates as
\begin{eqnarray*}
  \delta \mathcal{U} &=&(\alpha C)^{-1}r_{\mathcal{U}}, \\
 \delta p &=&D^{-1}(B\delta\mathcal{U}-r_{p}).
\end{eqnarray*}
%


Note that, in all these relaxation schemes, the key point is to find a good approximation, $C$, to $A$, and we particularly need   $C^{-1}$. This inspires us to consider an approximation of $C^{-1}$ directly. In our recently work \cite{CH2021addVanka}, we have shown that the mass matrix $Q$ derived from bilinear finite elements is a very good approximation to the inverse of scalar Laplacian discretized by finite difference method. Therefore,  we will consider 
\begin{equation}\label{eq:C-inverse-Q}
C^{-1}  = 
\begin{pmatrix}
Q &  0\\
0 & Q
\end{pmatrix}
\end{equation}
for the above three relaxation schemes. We call the resulting schemes  {\em mass-based relaxations}, for example, mass-based Braess-Sarazin relaxation. 
%
%


\section{Smoothing analysis}\label{sec:smoothing-analysis} 
LFA \cite{MR1807961} is a very useful tool for analyzing the convergence speed of multigrid methods, by examining the spectral radius of {\it symbol} of the underlying operator.  In general, the smoothing factor, $\mu$, of LFA  offers  sharp predictions of actual multigrid convergence. In this work, we employ LFA to examine the block-structured  multigrid relaxation schemes proposed in Section \ref{sec:discretization} by exploring the LFA smoothing factor.  We consider geometric multigrid methods for staggered discretizations with standard coarsening, that is, we construct a sequence of coarse grids by doubling the mesh size in each spatial direction.  High and low frequencies for standard coarsening are given by
\begin{equation*}
  \boldsymbol{\theta}\in T^{{\rm low}} =\left[-\frac{\pi}{2},\frac{\pi}{2}\right)^{2}, \, \boldsymbol{\theta}\in T^{{\rm high}} =\displaystyle \left[-\frac{\pi}{2},\frac{3\pi}{2}\right)^{2} \bigg\backslash \left[-\frac{\pi}{2},\frac{\pi}{2}\right)^{2}.
\end{equation*}
In the following, $\widetilde{L}$ denotes the symbol of operator $L$.  We omit the details on how to compute the symbols. For more details, see \cite{MR1807961} and \cite{he2018local}.
\begin{definition} The smoothing factor for  the error-propagation for the relaxation scheme $\mathcal{S}$, see \eqref{eq:Sh-general-form}, is defined as
\begin{equation*}
  \mu_{{\rm loc}}(\boldsymbol{p})=\max_{\boldsymbol{\theta}\in T^{{
  \rm high}}}\left\{\left|\lambda(\widetilde{\mathcal{S}}(\boldsymbol{\theta}))\right| \,\,\right\},
\end{equation*}
where  $\boldsymbol{p}$ is algorithmic parameters,  and $\lambda\big(\widetilde{\mathcal{S}}(\boldsymbol{\theta})\big)$ denotes the  eigenvalue of symbol $\widetilde{\mathcal{S}}(\boldsymbol{\theta})$. 
\end{definition}
Often, one can minimize $\mu_{{\rm loc}}(\boldsymbol{p})$ with respect to $\boldsymbol{p}$ to obtain fast convergence speed.  Thus, we define the optimal smoothing factor as follows.
\begin{definition}\label{def-opt}
We define the optimal smoothing factor  for  the error-propagation for the relaxation scheme $\mathcal{S}$  as
  \begin{equation*}
    \mu_{{\rm opt}}=\min_{\boldsymbol{p} \in \Theta}{\mu_{{\rm loc}}}(\boldsymbol{p}),
  \end{equation*}
  where $\Theta$ is the set of allowable parameters. 
\end{definition}
We note that $\boldsymbol{p} $ may contains many components depending on the selection of the relaxation scheme. In our setting, $\boldsymbol{p} $ has dimension one, two or three.


%
The symbol of operator $\mathcal{L} $, see \eqref{eq:Lh-operator}, is given by \cite{he2018local}
\begin{equation*}\label{Symbol-exact}
  \widetilde{\mathcal{L}}(\theta_1,\theta_2) =\frac{1}{h^2}\begin{pmatrix}
       4m(\boldsymbol{\theta})& 0 & i 2h \sin\frac{\theta_1}{2}  \\
      0 &  4m(\boldsymbol{\theta})& i 2h \sin\frac{\theta_2}{2} \\
      -i 2h \sin\frac{\theta_1}{2}   &  -i 2h \sin\frac{\theta_2}{2} &0
    \end{pmatrix},
\end{equation*}
where  $i^2=-1$ and 
\begin{equation}\label{eq:scaled-symbol-Ah}
m(\boldsymbol{\theta})=\frac{4-2\cos\theta_{1}-2\cos\theta_2}{4}=\sin^{2}(\theta_1/2)+\sin^{2}(\theta_2/2).
\end{equation}
Let $A_s$ be the scalar Laplacian operator. It is easily seen that  the symbol of the scalar Laplacian is
\begin{equation}\label{eq:scalar-symbol-Laplace}
\widetilde{A}_s = \frac{4m(\boldsymbol{\theta})}{h^2}.
\end{equation}
 The error-propagation symbol for a relaxation scheme, $\mathcal{S}$, applied to MAC scheme is
\begin{equation*}
\widetilde{ \mathcal{S}}(\boldsymbol{p},\boldsymbol{\theta})=I-\omega \widetilde{ M}^{-1}\widetilde{\mathcal{L}},
\end{equation*}
where $\boldsymbol{p}$ represents parameters within $M$ and  $\omega$.

Before we analyze the smoothing factor with $C^{-1}$  being the vector mass matrix defined in \eqref{eq:C-inverse-Q}, we   present some properties for the mass operator $Q$.  Recall that the mass matrix stencil using bilinear finite elements in two dimensions is 
\begin{equation}\label{eq:Ms2d}
Q= \frac{h^2}{36}
\begin{bmatrix}
 1& 4 & 1\\
4 & 16 & 4 \\
1 & 4 &1
\end{bmatrix}.
\end{equation} 
It   easily follows that the symbol of $Q$ is  $\widetilde{Q} =\frac{h^2}{9}(2+\cos\theta_1)(2+\cos \theta_2)$. 
 Let us denote 
\begin{equation} \label{eq:scaled-symbol-Q-inverse}
m_s(\boldsymbol{\theta}) =\left(\frac{1}{9}(2+\cos\theta_1)(2+\cos \theta_2)\right)^{-1}, 
\end{equation}
which will be   useful  for our discussion later.  Using  \eqref{eq:scalar-symbol-Laplace} and \eqref{eq:scaled-symbol-Q-inverse}, we have
\begin{equation}\label{eq:ratio-mr}  
\widetilde{ Q} \widetilde{A}_s = \frac{4m(\boldsymbol{\theta})}{m_s(\boldsymbol{\theta}) }=: m_r(\boldsymbol{\theta}).
\end{equation}
\begin{lemma} \cite{CH2021addVanka}\label{lem:range-mr}
For $\boldsymbol{\theta} \in T^{high}$, $m_r(\boldsymbol{\theta}) \in [8/9, 16/9]=:[\eta_1,\eta_2]$.
\end{lemma}

We restate the optimal smoothing factor for  solving the Poisson problem using the mass-based relaxation \cite{CH2021addVanka}. 
\begin{lemma}\label{lemma:mass-smoothing-factor}
If we consider the mass stencil $Q$ to approximate the inverse of the scalar  Laplacian, $A_s$, and  $S_{s}=I-\omega QA_s$, then  the optimal smoothing factor for $S_{s}$  is
\begin{equation}\label{eq:smoothing-mass-Laplace}
  \mu_{s,\rm opt}= \min_{\omega} \max_{\boldsymbol{\theta}\in T^{{
  \rm high}}} {|1- \omega \widetilde{ Q}  \widetilde{A}_s|}=\frac{1}{3},
\end{equation}
where the minimum is achieved if and only if $\omega=\omega_{\rm opt} = \frac{3}{4}$. 
\end{lemma}
Lemma \ref{lemma:mass-smoothing-factor}  plays an important role in analyzing  the   optimal smoothing factor for the mass-based relaxation schemes for the Stokes equations because block relaxation schemes  include  $S_{s}$  as a piece of the overall relaxation.  Therefore, the  value $\frac{1}{3}$  offers a lower bound on the optimal smoothing factor for the  mass-based relaxation schemes. To  distinguish  the mass-based Braess-Sarazin relaxation,  distributive relaxation, and Uzawa relaxation, subscripts, such as $B, D$ and $U$,  will be used  in $S, \omega$ and $\mu_{\rm opt}$  throughout the rest of this paper.
\subsection{Mass-based Distributive relaxation}
  The symbol of operator $\mathcal{K}=\mathcal{L}P$, see \eqref{DWJ-system},  is given by
\begin{equation*}\label{Symbol-exact}
  \widetilde{\mathcal{K}}(\theta_1,\theta_2) =\frac{1}{h^2}\begin{pmatrix}
       4m(\boldsymbol{\theta})& 0 & 0  \\
      0 &  4m(\boldsymbol{\theta})& 0 \\
      -i 2h \sin\frac{\theta_1}{2}   &  -i 2h \sin\frac{\theta_2}{2} & 4m(\boldsymbol{\theta})
    \end{pmatrix}.
\end{equation*}
The symbol of the block relaxation operator \eqref{DWJ-Precondtion} with $C^{-1}$ given by the mass approximation \eqref{eq:C-inverse-Q} is
  \begin{equation*}\label{Symbol-PrecondDWJ}
   \widetilde{M}_{D}(\theta_1,\theta_2) =\frac{1}{h^2}\begin{pmatrix}
       m_s  & 0 & 0  \\
      0 &   m_s & 0 \\
      -i 2h \sin\frac{\theta_1}{2}   & -i 2h \sin\frac{\theta_2}{2} &m_s
    \end{pmatrix}.
\end{equation*}
It can be easily  shown  that all of the eigenvalues of the error-propagation symbol, $\mathcal{\widetilde{S}}_{D}(\alpha_{D}, \omega_{D},\boldsymbol{\theta})=I- \omega_{D}\widetilde{P} \widetilde {M}_{D}^{-1}\widetilde{\mathcal{L}}$, are
$1-\omega_{D}\frac{4m(\boldsymbol{\theta})}{m_s(\boldsymbol{\theta})}.$
\begin{theorem}\label{QDRsmoothing-theorem}
The optimal smoothing factor for $Q$-DR is
  \begin{equation*}\label{EX-smoothing}
   \mu_{{\rm opt},D}=\min_{ \omega_{D}}\max_{ \boldsymbol{\theta}\in T^{{\rm high}}}{\big|\lambda(\mathcal{\widetilde{S}}_{D}(\alpha_{D}, \omega_{D},\boldsymbol{\theta}))\big|}=\frac{1}{3},
\end{equation*}
where the minimum is uniquely achieved at  $\omega_{D}=\frac{3}{4}$.
\end{theorem}
 \begin{proof}
Note that $\frac{4m }{m_s }=\widetilde{Q}\widetilde{A}_s$.  As shown in Lemma \ref{lemma:mass-smoothing-factor}, we know that $\mu_{\rm opt}(I- \omega_{D}\widetilde{Q}\widetilde{A}_s)=\frac{1}{3}$  with $\omega_D=\frac{3}{4}$.  Thus, 
  $ \displaystyle \min_{ \omega_{D}} \max_{\boldsymbol{\theta}\in T^{{\rm high}}}\big|\lambda(\mathcal{\widetilde{S}}_{D}(\omega_{D},\boldsymbol{\theta}))\big|=\frac{1}{3}.$
 \end{proof}
Compared to Lemma \ref{lemma:mass-smoothing-factor}, Theorem \ref{QDRsmoothing-theorem}  shows  that $\frac{1}{3}$  is preserved for the block relaxation scheme.  A similar property is also  observed in \cite{he2018local} for the distributive weighted Jacobi relaxation for the Stokes equations.
 
 \begin{remark}
We point out some existing results for  variants of distributive relaxation  for MAC discretization of the Stokes equations.  For example,  LFA predicts a two-grid convergence factor for DGS which requires a single sweep of GS on the velocity equations plus one on the pressure unknowns, of 0.4 using 6-point interpolation and 12-point restriction in \cite{MR1049395}.  In \cite{he2018local},   LFA predicts a two-grid convergence factor for DWJ of 0.6  for both  linear and bilinear interpolation.  
\end{remark}

\subsection{Mass-based Braess-Sarazin-type relaxation schemes }
 
Next, we examine  Braess-Sarazin-type algorithms, which  were originally developed as a relaxation scheme for the Stokes equations \cite{braess1997efficient}.   In practice, (\ref{solution-of-precondtion}) is not solved exactly, see \cite{he2018local}. Thus, in the following, we first consider  exact solve for (\ref{solution-of-precondtion}), then inexact one.
%
%
 
 {\bf Exact Braess-Sarazin relaxation}: In \eqref{Precondtion}, for simplicity, we consider  $\alpha=1$ and $C^{-1}$ being the vector mass matrix, see \eqref{eq:C-inverse-Q}. Later,
  we will see  that the choice of $\alpha>0$ does not affect the optimal smoothing factor.  We first derive the corresponding optimal smoothing factor for the exact mass-based Braess-Sarazin relaxation.   We  refer to the relaxation scheme as  $Q$-BSR. 

\begin{theorem}\label{EX-smoothing-theorem}
The optimal smoothing factor for  $Q$-BSR  is
  \begin{equation*}\label{EX-smoothing}
   \mu_{{\rm opt},B}=\displaystyle \min_{\omega_{B}}\max_{ \boldsymbol{\theta}\in T^{{\rm high}}}{\big|\lambda(\mathcal{\widetilde{S}}_{B}( \omega_{B},\boldsymbol{\theta}))\big|}=\frac{1}{3},
\end{equation*}
where the minimum is uniquely achieved at $\omega_{B}=\frac{3}{4}$.
\end{theorem}
%
 \begin{proof}
Using  \eqref{eq:scaled-symbol-Q-inverse},  the symbol of $M_{B}$ is  
\begin{equation*}\label{Symbol-PrecondM}
   \widetilde{M}_{B}(\theta_1,\theta_2) =\frac{1}{h^2}\begin{pmatrix}
       m_s & 0 & i 2h \sin\frac{\theta_1}{2}  \\
      0 &  m_s & i 2h \sin\frac{\theta_2}{2} \\
      -i 2h \sin\frac{\theta_1}{2}   &  -i 2h \sin\frac{\theta_2}{2} &0
    \end{pmatrix}.
\end{equation*}
The symbol of the error-propagation matrix for   $Q$-BSR  is $\mathcal{\widetilde{S}}_{B}( \omega_{B},\boldsymbol{\theta})=I- \omega_{B} \widetilde {M}_{B}^{-1}\widetilde{\mathcal{L}}$. 
To minimize the spectral radius of $\mathcal{\widetilde{S}}_{B}$ with respect to high frequencies, we first compute the eigenvalues of $\widetilde {M}_{B}^{-1}\widetilde{\mathcal{L}}$.   It can be easily shown that  the determinant of $\widetilde{L}-\lambda \widetilde{M}_{B}$ is 
\begin{equation*}
 4m m_s \big(\lambda-1\big)^2\left (\lambda-\frac{4 m }{m_s}\right).
\end{equation*}
Thus, the eigenvalues of $\widetilde { M}_{B}^{-1}\mathcal{\widetilde{ L}}$ are  $1, 1$,  and $\displaystyle\frac{4m }{m_s }.$

Note that $\frac{4m }{m_s }=\widetilde{Q}\widetilde{A}_s$.  As in Lemma \ref{lemma:mass-smoothing-factor}, we know that $\mu_{\rm opt}(I- \omega_{B}\widetilde{Q}\widetilde{A}_s)=\frac{1}{3}$  with $\omega_B=\frac{3}{4}$. For the eigenvalue 1, we have  $|1-\omega_B 1|<\frac{1}{3}$, where $\omega_B=\frac{3}{4}$. This means that 
  $\displaystyle \min_{\omega_{B}}  \max_{\boldsymbol{\theta}\in T^{{\rm high}}} \big|\lambda(\mathcal{\widetilde{S}}_{B}( \omega_{B},\boldsymbol{\theta}))\big|=\frac{1}{3}$ with $\omega_B=\frac{3}{4}$. 
\end{proof}
We stress that the  optimal smoothing factor $\frac{1}{3}$ is much smaller than  $\frac{3}{5}$  using $C={\rm diag}(A)$ discussed in \cite{he2018local}.  It is interesting that  $\frac{1}{3}$ is also the optimal smoothing factor of BSR for the Stokes equations discretized by stabilized $Q_1-Q_1$ elements \cite{he2019local}.

\begin{remark}
One can consider weighted mass approximations, that is, $\alpha $ is varying in \eqref{Precondtion}. However, from the proof of Theorem \ref{EX-smoothing-theorem}, the scaling does not affect the optimal smoothing factor, which only changes the optimal damping parameter $\omega_B$.
\end{remark}

 {\bf Inexact Braess-Sarazin  relaxation}:  The exact  Braess-Sarazin algorithm requires an exact inversion of the Schur complement \eqref{solution-of-precondtion}, which is very expensive.  Our previous work \cite{he2018local} shows that   inexact BSR can achieve the  same convergence factor as the exact BSR  for MAC discretization of the Stokes equations. 
Therefore, we explore mass-based  inexact Braess-Sarazin relaxation, using a single sweep of weighted Jacobi iteration with weight $ \omega_J$ to approximate the solution of \eqref{solution-of-precondtion}. We are asking whether the inexact version is able  to achieve the same smoothing factor of $\frac{1}{3}$. Here, we do not  pursue theoretical analysis. Instead, we numerically  check the performance of the inexact BSR with the requirement $\frac{\omega_B}{\alpha_B}=\frac{3}{4}$ because $1-\frac{\omega_B}{\alpha_B}\frac{m}{m_s}$ is the eigenvalue for both exact and inexact version. We report the results in Section \ref{sec:Numer}.  The results show that inexact   BSR   of a two-grid method can obtain the same convergence  factor, $\frac{1}{3} $, with   $\alpha_B=1.4, \omega_B=\frac{3}{4}\alpha_B$ and $ \omega_J=1$. Since the inexact Braess-Sarazin relaxation is simple to implement, avoiding  computing the exact inversion of the Schur complement,  we recommend inexact mass-based Braess-Sarazin relaxation.

\begin{remark}

We mention that  \cite{he2018local} has studied the  $C$ in \eqref{Precondtion} to be (symmetric) Gauss-Seidel relaxation ((S)GS) for the MAC discretization of the Stokes equations, called (S)GS-BSR. For GS-BSR,  LFA  predicts the  optimal smoothing factor of 0.50 and  the two-grid convergence factor of 0.45 with optimal
parameters. For SGS-BSR, the optimal smoothing factor is  0.25 and the LFA-predicted two-grid
convergence factor is  0.20 with optimal parameters. For $C ={\rm diag}(A)$,   LFA predicts  the optimal smoothing factor of 0.60 for both exact and  inexact BSR versions, where  two-grid convergence factor is the same as optimal smoothing factor.
\end{remark}

\subsection{Mass-based Uzawa-type relaxation}\label{Uzawa-type}

 In this subsection, we consider   mass-based $\sigma$-Uzawa ($Q$-$\sigma$-Uzawa ), that is,  the inverse of $C$ in \eqref{Uzawa-Precondtion} is defined as \eqref{eq:C-inverse-Q} and  $D =\sigma^{-1}I$.   The symbol of  $M_{U}$ is given by
\begin{equation*} 
   \widetilde{M}_{U}(\theta_1,\theta_2) =\frac{1}{h^2}\begin{pmatrix}
        \alpha_{U}m_s & 0 & 0  \\
      0 &   \alpha_{U}m_s  & 0 \\
      -i 2h \sin\frac{\theta_1}{2}   &  -i 2h \sin\frac{\theta_2}{2} &-\sigma^{-1}h^2
    \end{pmatrix}.
    \end{equation*}
Furthermore, the determinant of $\widetilde{\mathcal{L}}-\lambda \widetilde{M}_{U}$ is 
\begin{equation*}
\frac{ (\alpha_{U}m_s)^{2}}{\sigma}  \big(\lambda-\frac{4m(\boldsymbol{\theta})}{m_s\alpha_{U}}\big)\bigg(\lambda^{2}-\frac{1+\sigma}{\alpha_{U}m_s}4m(\boldsymbol{\theta})\lambda+\frac{4m(\boldsymbol{\theta})\sigma}{\alpha_{U}m_s}\bigg) .
\end{equation*}
Using $m_r(\boldsymbol{\theta})=\frac{4m}{m_s}$ defined in \eqref {eq:ratio-mr}, we can rewrite the above determinant  as
\begin{equation}\label{eq:sigma-Uzawa-roots}
 \frac{ (\alpha_{U}m_s)^{2}}{\sigma} \big(\lambda-\frac{m_r}{ \alpha_{U}}\big)\bigg(\lambda^{2}-\frac{(1+\sigma)m_r }{\alpha_{U} }\lambda+\frac{m_r\sigma}{\alpha_{U} }\bigg),
\end{equation}
which will be more convenient for a later discussion. We follow  \cite{he2018local} to carry out  the analysis of Uzawa relaxation scheme.

From \eqref{eq:sigma-Uzawa-roots}, we see one eigenvalue of $\widetilde{M}_{U}^{-1}\widetilde{\mathcal{L}}$ is  $\lambda_{*}:=\frac{m_r}{\alpha_{U}}$. From Lemma \ref{lem:range-mr} and Lemma \ref{lemma:mass-smoothing-factor}, we know the optimal smoothing factor for these modes is  
\begin{equation*} 
  \left|1-\frac{\omega_{U}}{\alpha_{U}}\eta_1\right|=\left|1-\frac{\omega_{U}}{ \alpha_{U}}\eta_2\right|=\frac{1}{3},
\end{equation*}
with $\frac{\omega_{U}}{\alpha_{U}}=\frac{3}{4}$.

Next, we explore the optimal smoothing factor for other two eigenvalues of $\widetilde{M}_{U}^{-1}\widetilde{\mathcal{L}}$ in  \eqref{eq:sigma-Uzawa-roots}. We denote by $\lambda_{1},\lambda_{2}$ the roots of
\begin{equation}\label{quadratic-function-root-U}
  d(\lambda)=\lambda^{2}-\frac{(1+\sigma)m_r(\boldsymbol{\theta})}{\alpha_{U}}\lambda+\frac{m_r(\boldsymbol{\theta})\sigma}{\alpha_{U}}.
\end{equation}
The discriminant of $d(\lambda)$  is 
\begin{equation*} 
  \Delta (\alpha_{U},\sigma)=\frac{m_r(\boldsymbol{\theta})(1+\sigma)^{2}}{\alpha_{U}^{2}}\bigg(m_r(\boldsymbol{\theta})-\frac{4\alpha_{U}\sigma}{(1+\sigma)^{2}}\bigg).
\end{equation*}
Let the two roots of $\Delta=0$ with respect to $m_r$ be
\begin{equation*} 
  m_{1}=0,\quad m_{2}=\frac{4\alpha_{U}\sigma}{(1+\sigma)^{2}}.
\end{equation*}

From  \eqref{quadratic-function-root-U}, we have
\begin{eqnarray}
  \lambda_{1}+\lambda_{2}&=&\frac{m_r(\boldsymbol{\theta})(1+\sigma)}{\alpha_{U}}, \label{U-roots-sum} \\
  \lambda_{1}\lambda_{2}&=&\frac{m_r(\boldsymbol{\theta})\sigma}{\alpha_{U}}, \label{U-roots-times}\\
 \lambda_{1,2}&=&\frac{(1+\sigma)m_r(\boldsymbol{\theta})}{2\alpha_{U}}\left(1\pm \sqrt{1-\frac{m_2}{m_r(\boldsymbol{\theta})}}\right).\label{roots-formulation}
\end{eqnarray}
The value $m_2$ plays an important role in determining  the optimal smoothing factor with respect to eigenvalues $\lambda_1$ and $\lambda_2$.  There are three cases for  $\lambda_1$ and $\lambda_2$: 
\begin{enumerate}
\item only complex eigenvalues, that is, $m_2>\max{m_r}=\eta_2$.
\item only real eigenvalues, that is, $m_2< \min{m_r}=\eta_1$.
\item giving both real and complex eigenvalues, that is, $\frac{8}{9}=\eta_1\leq m_2\leq \eta_2=\frac{16}{9}$.
\end{enumerate}
 To minimize the smoothing factor, we explore the optimal smoothing factor for the above three cases. Then, we choice the best one.

In order to discuss the complex eigenvalues, we set  $\Upsilon (m_r(\boldsymbol{\theta}))$ to be the magnitude of   $1-\omega_{U}\lambda_{1}$ at frequency
 $\boldsymbol{\theta}$. Using \eqref{U-roots-sum} and \eqref{U-roots-times} gives
\begin{align*}
 \Upsilon^{2}(m_r(\boldsymbol{\theta})) &= (1-\omega_{U}\lambda_{1})(1-\omega_{U}\lambda_{2}) \\
              &=1-(\lambda_1+\lambda_2)\omega_{U}+\lambda_1\lambda_2\omega_{U}^{2} \\
             &=1+\frac{\omega_{U}}{\alpha_{U}}(\omega_{U}\sigma-\sigma-1)m_r(\boldsymbol{\theta}).  
\end{align*}
%
\begin{theorem}\label{thm:complex-max-value}
Assume that $m_{2}\geq \frac{8}{9}$ and let $\gamma=\min\{m_{2},16/9 \}$. For $m_r(\boldsymbol{\theta})\in[8/9, \gamma]$, eigenvalues $\lambda_{1}$ and $\lambda_{2}$ are complex conjugates and the smoothing factor for these modes is
  \begin{equation*}\label{general-complex-smoothing}
    \mu^C=\max_{m_r(\boldsymbol{\theta})\in[8/9,\gamma]} \Upsilon(m_r(\boldsymbol{\theta}))=\sqrt{1+\frac{8\omega_U(\omega_U\sigma-\sigma-1)}{9\alpha_U}}\geq\sqrt{1-\frac{8}{9\gamma}},
  \end{equation*}
where the equality is achieved  provided that 
  \begin{equation*} 
\frac{\omega_{U}}{\alpha_{U}}(\omega_U\sigma-\sigma-1)=-\frac{1}{\gamma}.
\end{equation*}
\end{theorem}
\begin{proof}
When $m_r(\boldsymbol{\theta})\in[8/9, \gamma]$, $\Delta(\alpha_{U},\sigma)\leq 0$ and $|1-\omega_{U}\lambda_{1}|=|1-\omega_{U}\lambda_{2}|=\Upsilon(m_r(\boldsymbol{\theta}))$.  To guarantee convergence,  we require $\Upsilon^2(m_r(\boldsymbol{\theta}))<1$, which leads to  $\frac{\omega_U(\omega_U\sigma-\sigma-1)}{\alpha_U}<0$. Since $\gamma=\min\{m_2,16/9\}$, we have
\begin{eqnarray*}
\Upsilon^2(\gamma)&=&1+\frac{\omega_{U}}{\alpha_{U}}(\omega_{U}\sigma-\sigma-1)\gamma\\
&\geq&1+\frac{\omega_{U}}{\alpha_{U}}(\omega_{U}\sigma-\sigma-1)m_2\\
&=&\left(1-\frac{2\omega_U\sigma}{1+\sigma}\right)^2 \geq0.
\end{eqnarray*}
It follows that
\begin{equation*} 
\frac{\omega_{U}}{\alpha_{U}}(\omega_U\sigma-\sigma-1)\geq-\frac{1}{\gamma}.
\end{equation*}
Furthermore,
\begin{equation*}
  \max_{m_r(\boldsymbol{\theta})\in[8/9,\gamma]}\Upsilon(m_r(\boldsymbol{\theta}))=\Upsilon(8/9)=\sqrt{1+\frac{8\omega_{U}}{9\alpha_{U}}(\omega_U\sigma-\sigma-1)}
  \geq\sqrt{1-\frac{8}{9\gamma}},
\end{equation*}
where the equality is achieved if and only if $\frac{\omega_U(\omega_U\sigma-\sigma-1)}{\alpha_U}=\frac{-1}{\gamma}$.
\end{proof}
 
 Based on the above results, we are ready to give a lower bound on the smoothing factor for the case that $m_2>\eta_2=\frac{16}{9}$.
\begin{theorem}\label{m2>2}
If $m_{2}=\frac{4\alpha_{U}\sigma}{(1+\sigma)^{2}}> \eta_2=\frac{16}{9}$, then the optimal smoothing factor for  $Q$-$\sigma$-Uzawa relaxation  is larger than  $\frac{\sqrt{2}}{2}$.
\end{theorem}
\begin{proof}
  From Theorem \ref{thm:complex-max-value} with $\gamma =\frac{16}{9}$, we know the smoothing factor for the complex modes is 
  \begin{equation}\label{eq:large-than-root-half}
  \mu^C=\Upsilon(8/9)\geq\sqrt{1-\frac{8}{9\gamma}}=\sqrt{1-\frac{8}{9 \cdot \frac{16}{9} } }  =\frac{\sqrt{2}}{2},
\end{equation}
where the equality is achieved on condition that $\Upsilon^2(\eta_2)=0$.  Next, we   show that $\Upsilon^2(\eta_2)>0$. 
 We proceed to prove this by contradiction.  Supposing  that $\Upsilon^2(\eta_2)=1+\frac{\omega_{U}}{\alpha_{U}}(\omega_{U}\sigma-\sigma-1)\frac{16}{9}=0$, we have $\frac{\alpha_{U}}{\omega_{U}(\sigma+1-\omega_{U}\sigma)}=\frac{16}{9}$.  Using $m_2=\frac{4\alpha_{U}\sigma}{(1+\sigma)^{2}}>\frac{16}{9}$ gives
  \begin{equation*}
    \frac{4\alpha_{U}\sigma}{(1+\sigma)^{2}}> \frac{\alpha_{U}}{\omega_{U}(\sigma+1-\omega_{U}\sigma)},
  \end{equation*}
  which is equivalent to 
  \begin{equation*}
    \bigg(\frac{\omega_{U}\sigma}{1+\sigma}-\frac{1}{2}\bigg)^2<0.
  \end{equation*}
 Thus, $\mu^C$ is large than $\frac{\sqrt{2}}{2}$  in \eqref{eq:large-than-root-half}.
\end{proof}

Next, we explore  the case where $m_2\leq \eta_2=\frac{16}{9}$. For $m_r(\boldsymbol{\theta})\in[m_2,16/9]$, the two roots, $\lambda_{1,2}$, are real. Using  (\ref{roots-formulation}) gives 
\begin{eqnarray*}
|1-\omega_{U}\lambda_1|&=&
\left|1-\frac{(1+\sigma)\omega_{U}}{2\alpha_{U}}m_r(\boldsymbol{\theta})\bigg(1+\sqrt{1-\frac{m_2}{m_r(\boldsymbol{\theta})}}\bigg)\right|,\\
|1-\omega_{U}\lambda_2|&=&
\left|1-\frac{(1+\sigma)\omega_{U}}{2\alpha_{U}}m_r(\boldsymbol{\theta})\bigg(1-\sqrt{1-\frac{m_2}{m_r(\boldsymbol{\theta})}}\bigg)\right|.
\end{eqnarray*}
Define
\begin{eqnarray*}
  \chi_{+}(m_r(\boldsymbol{\theta})) &=& \frac{m_r(\boldsymbol{\theta})}{2}\left(1+\sqrt{1-\frac{m_2}{m_r(\boldsymbol{\theta})}}\right),\label{root-positive} \\
\chi_{-}(m_r(\boldsymbol{\theta})) &=& \frac{m_r(\boldsymbol{\theta})}{2}\left(1-\sqrt{1-\frac{m_2}{m_r(\boldsymbol{\theta})}}\right).\label{root-negative}
\end{eqnarray*}
It can  be shown that function $\chi_{+}(m_r(\boldsymbol{\theta}))$ is an increasing function of $m_r(\boldsymbol{\theta})$ for $m_r(\boldsymbol{\theta})\in[m_2,16/9]$, and $\chi_{+}(m_r(\boldsymbol{\theta}))$ is a  decreasing function, see  \cite{he2018local}.  Thus, 
\begin{eqnarray*}
  &\chi_{+}(m_r(\boldsymbol{\theta}))_{\rm max}&=\chi_{+}(16/9)=\frac{8}{9}\left(1+\sqrt{1-\frac{9 m_2}{16}}\right)=:\chi_{1},\label{R+Max}\\
  & \chi_{-}(m_r(\boldsymbol{\theta}))_{\rm min}&=\chi_{-}(16/9)=\frac{8}{9}\left(1-\sqrt{1-\frac{9m_2}{16}}\right) =:\chi_{2}.\label{R-Min}
\end{eqnarray*}
It follows that the smoothing factor for the two real eigenvalues in this case is
\begin{eqnarray*}
 & \mu^R &:=\max_{\boldsymbol{\theta}\in T^{{\rm high}}}\big|\lambda(\mathcal{\widetilde{S}}_{U}(\alpha_{U},\omega_{U},\sigma,\boldsymbol{\theta}))\big|, \\
  &&={\rm max}\bigg\{\big|1-\frac{(1+\sigma)\omega_U}{\alpha_U}\chi_1\big|,\big|1-\frac{(1+\sigma)\omega_U}{\alpha_U}\chi_2\big|\bigg\}.
\end{eqnarray*}
The above expression can be further simplified to
\begin{equation}\label{compare-R1-R2}
\mu^R=\left\{
  \begin{aligned}
    &\frac{(1+\sigma)\omega_{U}}{\alpha_{U}}\chi_1-1, \quad {\rm if}\quad \frac{(1+\sigma)\omega_{U}}{\alpha_{U}}\geq \frac{9}{8}. \\
    &1-\frac{(1+\sigma)\omega_{U}}{\alpha_{U}}\chi_2, \quad {\rm if}\quad \frac{(1+\sigma)\omega_{U}}{\alpha_{U}}\leq \frac{9}{8}.
  \end{aligned}
               \right. \\
  \end{equation}
We are now able to give a lower bound on the optimal smoothing factor for the case where $m_2\leq\frac{8}{9}$.
\begin{theorem}\label{U-all-complex}
If $m_{2}=\frac{4\alpha_{U}\sigma}{(1+\sigma)^{2}}\leq\frac{8}{9}$, then the optimal smoothing factor for $Q$-$\sigma$-Uzawa  relaxation is not less than $\frac{\sqrt{2}}{2}$.
\end{theorem}
\begin{proof}
When $m_2\leq\frac{8}{9}$, $\lambda_1, \lambda_2$ are all real.  From  (\ref{compare-R1-R2}), the smoothing factor for $m_r(\boldsymbol{\theta})\in[8/9,16/9]$ is given by
  \begin{equation*}
  \mu^R=\left\{
  \begin{aligned}
    &\frac{8(1+\sigma)\omega_{U}}{9\alpha_{U}}(1+\frac{\sqrt{2}}{2})-1, \quad {\rm if}\quad \frac{(1+\sigma)\omega_{U}}{\alpha_{U}}\geq \frac{9}{8}. \\
    &1-\frac{8(1+\sigma)\omega_{U}}{9\alpha_{U}}(1-\frac{\sqrt{2}}{2}), \quad {\rm if}\quad \frac{(1+\sigma)\omega_{U}}{\alpha_{U}}\leq \frac{9}{8}.
  \end{aligned}
               \right. \\
  \end{equation*}
Note that when $\frac{(1+\sigma)\omega_{U}}{\alpha}=\frac{9}{8}$, $\mu^R$ achieves  its minimum value of $\frac{\sqrt{2}}{2}$.  However,  the conditions that $\frac{(1+\sigma)\omega_{U}}{\alpha}=\frac{9}{8}$ and $m_2=\frac{4\alpha_{U}\sigma}{(1+\sigma)^{2}}\leq\frac{8}{9}$ might not be satisfied at the same time. This implies that  the optimal smoothing factor may be larger than  $\frac{\sqrt{2}}{2}$.
\end{proof}
From previous  discussions, we know that  when $m_2>\frac{16}{9}$ or $m_2 \leq \frac{8}{9}$, the optimal smoothing factor is at least $\frac{\sqrt{2}}{2}$.  Next, we will show that  the global optimal smoothing factor for all choices of $m_2$ occurs   when $\frac{8}{9}\leq m_2\leq \frac{16}{9}$.  
\begin{theorem}\label{inside-optimal}
The optimal smoothing factor for $Q$-$\sigma$-Uzawa relaxation over all possible parameters is   
 \begin{eqnarray*}\label{mixed-optimal-factor}
  \mu_{{\rm opt}, U}&=&   \min _{(\alpha_{U},\omega_{U},\sigma)}\max_{ \boldsymbol{\theta}\in T^{{\rm high}}}\left\{\left|1-\omega_{U}\lambda_{*}\right|,\,  \mu^R,\, \mu^C\right\}
  =\sqrt{1-\frac{9m_{{\rm opt}}}{16 }}= \sqrt{\frac{1}{3}}\approx 0.577,
\end{eqnarray*}
where the minimum is attained provided that $m_2=m_{{\rm opt}}=\frac{32}{27}$ with
\begin{eqnarray*}
     \frac{1}{ 3 \mu_{{\rm opt},U} } \leq& \omega_{U}  & \leq \frac{2}{3(1-\mu_{{\rm opt},U})}, \\
   &\alpha_{U}&=\frac{8\omega^2_U}{3(3\omega_U-1)}, \\
  & \sigma&=\frac{1}{3 \omega_{U}-1}. 
\end{eqnarray*}
 \end{theorem}
\begin{proof}
We first consider $m_{2}\in[8/9, 16/9]$ and  $\frac{(1+\sigma)\omega_U}{\alpha_U}=\frac{9}{8}$. In this situation, the two expressions in (\ref{compare-R1-R2}) coincide and $m_2=\frac{4\alpha_U\sigma}{(1+\sigma)^2}=\frac{4\cdot 8^2}{9^2}\frac{\omega^2_U\sigma }{\alpha_U}$.

For $m_r(\boldsymbol{\theta})\in[m_2,16/9]$, from  (\ref{compare-R1-R2}), we have
\begin{equation*}
\mu^R=\frac{(1+\sigma)\omega_{U}}{\alpha_{U}}\chi_1-1=\sqrt{1-\frac{9 m_2}{16}}=\sqrt{1-\frac{16\omega^2_U\sigma}{9\alpha_U}}.
\end{equation*}
For $m_r(\boldsymbol{\theta})\in[8/9, m_2]$,  Theorem \ref{thm:complex-max-value} tells us that
\begin{equation*}
\mu^C=\sqrt{1+\frac{8 \omega_U(\omega_U\sigma-\sigma-1)}{9\alpha_U}}=\sqrt{ \frac{8\omega^2_U\sigma}{9\alpha_U}}.
\end{equation*}
Since $\mu^R$ is a decreasing function of $\frac{\omega^2_U\sigma}{\alpha_U}$ and $\mu^C$ is an increasing function of $\frac{\omega^2_U\sigma}{\alpha_U}$, the optimal smoothing factor over the modes bounded by these factors is achieved if and only if $\mu^R=\mu^C$ and is given by
\begin{equation*}
\mu_{{\rm opt}, U}=\displaystyle \min _{(\alpha_{U},\omega_{U},\sigma)}\max_{m_r(\boldsymbol{\theta})\in[8/9, 16/9]}
\left\{\sqrt{1- \frac{16\omega^2_U\sigma}{9\alpha_U}},\,\sqrt{ \frac{8\omega^2_U\sigma}{9\alpha_U}}\right\}=\sqrt{\frac{1}{3}},
\end{equation*}
provided that
\begin{eqnarray}
  \frac{\omega^2_U\sigma}{\alpha_U}&=&\frac{3}{8},\label{optimal-condion-1} \\
  \frac{(1+\sigma)\omega_U}{\alpha_U}&=&\frac{9}{8}.\label{optimal-condion-2}
\end{eqnarray}
Furthermore, $m_{{\rm opt}}:=m_2=\frac{4\cdot 8^2}{9^2} \frac{\omega^2_U\sigma}{\alpha_U}=\frac{32}{27}$. Since $\sqrt{\frac{1}{3}} < \frac{\sqrt{2}}{2}$, it means that $m_2 \in [8/9, 16/9] $ gives a better smoothing factor.

Next, we show the above choice, $\frac{(1+\sigma)\omega_U}{\alpha_U}=\frac{9}{8}$, gives the best possible bound over these two modes, then we go back to consider the eigenvalues $1-\omega_U\frac{m(\boldsymbol{\theta})}{\alpha_U}$.

Let  $x=\frac{(1+\sigma)\omega_U}{\alpha_U}$ and $y=\frac{\omega^2_U\sigma}{\alpha_U}$. Then, $m_2=\frac{4\alpha_U\sigma}{(1+\sigma)^2}=\frac{4y}{x^2}$.
 Assume that $\mu^C\leq\sqrt{\frac{1}{3}}$,  that is, 
 \begin{equation*}
 \sqrt{1+\frac{8 \omega_U(\omega_U\sigma-\sigma-1)}{9\alpha_U}}=\sqrt{1-\frac{8x}{9}+\frac{8y}{9}}\leq\sqrt{\frac{1}{3}},
 \end{equation*}
which leads to $y\leq x-\frac{3}{4}$.
 
Now, we check $\mu^R$  under the condition that $y\leq x-\frac{3}{4}$. If $x>\frac{9}{8}$,  from (\ref{compare-R1-R2}), we have
\begin{eqnarray*}
\mu^R&=&\frac{(1+\sigma)\omega_U}{\alpha_U} \frac{8}{9}\left(1+\sqrt{1-\frac{9m_2}{16}}\right)-1\nonumber\\
  &=&\frac{8}{9}\left(x+\sqrt{x^2-9y/4}\right)-1\nonumber\\
  &\geq& \frac{8}{9}  \sqrt{x^2-\frac{9}{4}(x-3/4)} \nonumber\\
  &=& \frac{8}{9} \sqrt{(x- 9/8)^2+27/64}\nonumber\\
  &>&\sqrt{\frac{1}{3}}.\nonumber
\end{eqnarray*}
This means that when $x>1$, the optimal smoothing factor is larger than $\sqrt{\frac{1}{3}}$.

If $x<\frac{9}{8}$,  from (\ref{compare-R1-R2}),  we have
\begin{eqnarray*}
\mu^R&=&1-\frac{(1+\sigma)\omega_U}{\alpha_U} \frac{8}{9}\left(1-\sqrt{1-\frac{9 m_2}{16 }}\right)\nonumber\\
  &=&1-\frac{8}{9} x+\frac{8}{9} \sqrt{x^2-9y/4}\nonumber\\
  &\geq& \frac{8}{9}  \sqrt{x^2-9/4\cdot (x-3/4)} \nonumber\\
  &=& \frac{8}{9} \sqrt{(x-9/8)^2+27/64}\nonumber\\
  &>&\sqrt{\frac{1}{3}}.\nonumber
\end{eqnarray*}
Consequently,  the optimal smoothing factor is larger than $\sqrt{\frac{1}{3}}$ for $x<\frac{9}{8}$.

Thus, over all the choices of $x$, the optimal smoothing factor for these modes ($\lambda_{1,2}$) is $\mu_{{\rm opt},U}=\sqrt{\frac{1}{3}}$ with $x=\frac{(1+\sigma_U)\omega_U}{\alpha_U}=\frac{9}{8}$.

Now, we need to include $\lambda_{*}=\frac{m_r(\boldsymbol{\theta})}{\alpha_U}$. For  $\lambda_{*}$,  we have shown that 
\begin{equation*}
 \min_{(\alpha_{U},\omega_{U},\sigma
)}\max_{ \boldsymbol{\theta}\in T^{{\rm high}}}\big|1-\frac{\omega_{U}}{\alpha_{U}}m_r(\boldsymbol{\theta})\big|=\frac{1}{3}< \sqrt{\frac{1}{3}}.
\end{equation*}
To obtain optimal result for all modes, we need
\begin{equation*}
  |1-\frac{8\omega_{U}}{9\alpha_U}|\leq\mu_{{\rm opt},U}\quad  {\rm and }\,\quad |1-\frac{16 \omega_{U}}{9\alpha_U}|\leq\mu_{{\rm opt},U},
\end{equation*}
that is,
\begin{equation}\label{third-cond}
\frac{9}{8}(1-\mu_{{\rm opt},U})\frac{1}{\omega_U}\leq\frac{1}{\alpha_U}\leq\frac{9(1+\mu_{{\rm opt},U}) }{16}\frac{1}{\omega_U}.
\end{equation}
Simplifying (\ref{optimal-condion-1}) and (\ref{optimal-condion-2}) gives 
\begin{eqnarray}
  \alpha_{U}&=&\frac{8\omega^2_U}{3(3\omega_U-1)},\label{parameter-condition-2}\\
   \sigma&=&\frac{1}{3 \omega_{U}-1}.\label{parameter-condition-3}
\end{eqnarray}
Substituting  (\ref{parameter-condition-2})  into  (\ref{third-cond}) leads to
\begin{equation}\label{parameter-condition-1}
  \frac{1}{ 3 \mu_{{\rm opt},U} }\leq \omega_{U}\leq \frac{2}{3(1-\mu_{{\rm opt},U})}.
\end{equation}
Note that parameters $\omega_U=1, \alpha_U=\frac{4}{3},\sigma=\frac{1}{2}$ satisfy \eqref{parameter-condition-2},  \eqref{parameter-condition-3}, and \eqref{parameter-condition-1}.
\end{proof}

We summarize some results for  Uzawa-type relaxation scheme for MAC discretization of the Stokes equations. In \cite{MR3217219}, LFA for GS-Uzawa predicts a two-grid convergence factor of 0.87 with 2 sweeps of GS used on the velocity block in each sweep of Uzawa. Moreover, LFA predictions for SGS-Uzawa show a smoothing factor of 0.5 when a single sweep of SGS is used, and an LFA-predicted two-grid convergence factor of 0.44 in \cite{MR3217219}. In \cite{he2018local}, LFA predicts two-grid convergence factor of  $\sqrt{\frac{3}{5}} \approx 0.775$ for $\sigma$-Uzawa. 
 
\begin{remark}
From the analysis of the optimal smoothing factors for the mass-based relaxation schemes, we see how important it is to choose the approximation $C$. One can consider   $C$ to be  additive Vanka approaches introduced in \cite{CH2021addVanka}. It might give a slightly better optimal convergence factor because the optimal smoothing factor for the element-wise Vanka is 0.280 applied to the Laplacian, see \cite{CH2021addVanka}.  
\end{remark}
\section{Numerical experiments}\label{sec:Numer}
In this section, we present   multigrid performance to confirm our  theoretical optimal  smoothing factors. 
 Here, we use the defects  $d_{h}^{(k)}=b-\mathcal{L}_h x_k (k=1,2,\cdots)$ to experimentally measure the actual convergence factor given by
${\rho}^{(k)}_{m}=\sqrt[k]{\frac{\|d_{h}^{(k)}\|_{2}}{\|d_{h}^{(0)}\|_{2}}}$ (see \cite{MR1807961}), with $k=100$.  For simplicity,  let $\rho_m={\rho}^{(100)}_{m}$. To test the multigrid convergence factor, we consider the homogeneous problem ($b = 0$) with  periodic boundary conditions and the discrete solution $x_{h} \equiv 0$. The initial guess $x^{(0)}$ is chosen randomly. The coarsest grid is a $4\times4$ mesh. Rediscretization is used to define the coarse-grid operator. For the restriction operator,  we consider 6-point restrictions for the velocity unknowns,  and a 4-point cell-centered restriction for the pressure unknowns.  For the prolongation of the corrections, we apply the corresponding adjoint operators multiplied by a factor of 4. For more details, we refer to \cite{MR1049395}.   Let $\nu_1$ and $\nu_2$ be the number of pre- and postsmoothing iterations, respectively, and $\nu=\nu_1+\nu_2$.  We report the measured  two-grid and   $V$-cycle (or $W$-cycle) performance with different meshsize $h$ and  different  numbers of smoothing steps for $Q$-DR, $Q$-IBSR and $Q$-$\sigma$-Uzawa. 
 
\subsection{Multigrid performance}
We first report the multigrid performance for $Q$-DR in Table~\ref{QDR-PBC}. The two-grid  numerical results show good agreement between predicted convergence  $\mu^{\nu}_{\rm opt}$ and the true performance for $\nu<4$, and there is a tiny degradation when  $\nu=4$. Furthermore, we display the results for $V$-cycle, seeing a tiny degradation for  $\nu=3$ and $\nu=4$.  However, this difference  can be ignored.

\begin{table}[H]
 \caption{Multigrid convergence factor for $Q$-DR with $\mu_{\rm opt}=\frac{1}{3}$}
\centering
\begin{tabular}{ l|  c c c c }
\hline
$\nu=\nu_1+\nu_2$  & 1    & 2   &3    & 4  \\
\hline
$\mu^{\nu}_{\rm opt}$                   &0.333     &0.111       & 0.037       &0.012       \\
\hline
\hline
    \multicolumn{5} {c}  {\footnotesize{Two-grid}} \\
\hline
$\rho_m, h=1/32$             &0.328      &0.109        & 0.038     &0.028         \\
 
$\rho_m, h=1/64$            &0.326       & 0.108       &0.038      &0.030     \\
\hline
 
    \multicolumn{5} {c}  {\footnotesize{V-cycle}} \\
\hline
$\rho_m, h=1/128$            &0.324      &0.108        &0.053      & 0.041        \\
 
$\rho_m, h=1/256$            &0.324      &0.108        &0.053      &0.041        \\
\hline
\end{tabular}\label{QDR-PBC}
\end{table}

Table~\ref{IBSR-PBC} shows  the measured convergence factors for $Q$-IBSR with $\alpha_B=1.4, \omega_B=\frac{3}{4}\alpha_B$, and $\omega_J=1$, where we found these parameters are typically best. For the two-grid method, the actual convergence factors agree with the LFA predicted results $\mu^{\nu}_{\rm opt}$, except for $\nu=4$.  Although we see degradation in performance for $V$-cycles   with $\nu>1$, this degradation can be mitigated by using $W$-cycles.
\begin{table}[H]
 \caption{Multigrid convergence factor for $Q$-IBSR with $\alpha_B=1.4, \omega_B=\frac{3}{4}\alpha_B,\omega_J=1$ and $\mu_{\rm opt}=\frac{1}{3}$}
\centering
\begin{tabular}{ l|c c c c }
\hline
$\nu=\nu_1+\nu_2$  & 1    & 2   &3    & 4  \\
\hline
$\mu^{\nu}_{\rm opt}$      &0.333     &0.111       & 0.037       &0.012       \\
\hline
 \hline
    \multicolumn{5} { c }  {\footnotesize{Two-grid}} \\
\hline
$\rho_m, h=1/32$             &0.323       &0.110      & 0.037       &0.027        \\
 
$\rho_m, h=1/64$             &0.326       &0.109      &0.037        &0.027        \\
\hline
 
    \multicolumn{5} { c }  {\footnotesize{V-cycle}} \\
\hline
$\rho_m, h=1/128$            &0.326       &0.127      &0.081      &0.062        \\
 
$\rho_m, h=1/256$            &0.326       &0.178      &0.105      &0.080        \\
\hline
 
    \multicolumn{5} { c }  {\footnotesize{W-cycle}} \\
\hline
$\rho_m, h=1/128$            &0.326       &0.109      &0.037      &0.027       \\

$\rho_m, h=1/256$            &0.326       &0.109     &0.037     &0.027        \\
\hline
\end{tabular}\label{IBSR-PBC}
\end{table}


In Table~\ref{Uzawa-PBC}, we show the multigrid performance for $Q$-$\sigma$-Uzawa.  We see that the  measured convergence factors of two-grid and $W$-cycles methods match well with the LFA predictions $\mu^{\nu}_{\rm opt}$. However,  there is  a large degradation in performance for  $V$-cycles. Therefore, in practice,  $W$-cycles are more robust.

\begin{table}[H]
 \caption{Multigrid convergence factor for $Q$-$\sigma$-Uzawa with $\alpha_U=1, \omega_U=\frac{4}{3},\sigma=\frac{1}{2}$ and $\mu_{\rm opt}=\sqrt{\frac{1}{3}}$}
\centering
\begin{tabular}{ l|c c c c }
\hline
$\nu=\nu_1+\nu_2$  & 1    & 2   &3    & 4  \\
\hline
$\mu^{\nu}_{\rm opt}$                   &0.577      &0.333       &0.193      &0.111       \\
\hline
 \hline
    \multicolumn{5} { c }  {\footnotesize{Two-grid}} \\

\hline
$\rho_m, h=1/32$            &0.562       &0.322       &0.187     &0.108        \\
 
$\rho_m, h=1/64$            &0.559      &0.321       &0.186      &0.107        \\
\hline
 
    \multicolumn{5} { c }  {\footnotesize{V-cycle}} \\
\hline
$\rho_m, h=1/128$            &0.558      &0.668       &0.401      &0.236        \\
 
$\rho_m, h=1/256$            &0.744       &0.932       &0.541     &0.303        \\
\hline
 
    \multicolumn{5} { c }  {\footnotesize{W-cycle}} \\
\hline
$\rho_m, h=1/128$            &0.558       &0.321       &0.186      &0.106        \\

$\rho_m, h=1/256$            &0.558       &0.321       &0.186      &0.107       \\
\hline
\end{tabular}\label{Uzawa-PBC}
\end{table}

To sum up, for the three mass-based relaxation methods,  we observe that the LFA predicted optimal smoothing factors agree with the measured two-grid convergence
factors. There is a tiny degradation in $V$-cycles performance   for $Q$-IBSR and $Q$-$\sigma$-Uzawa schemes, however,   $W$-cycles almost  have the same performance as  those of two-grid methods.  

\begin{remark}
Note that \cite{MR1049395} and \cite{he2018local} consider bilinear interpolation for velocity
(12pts) and pressure (16pts).  Here, we do not consider this choice because the current simple   interpolation works well. 
\end{remark}

\subsection{Cost comparison}

We summarize our   optimal smoothing factors for the mass-based relaxation schemes in Table~\ref{parameters-choosing}, including the results from \cite{he2018local} as a comparison.  Table \ref{parameters-choosing} provides a complete overview of the  smoothing  properties of our new relaxation schemes, which outperform these discussed in \cite{he2018local}.
 
Note that an estimate of the computational work is presented in \cite{he2018local} for multigrid methods with   distributive  relaxation,  inexact Braess-Sarazin,  and $\sigma$-Uzawa. The computational work in  \cite{he2018local} is true for the mass-based relaxation schemes proposed in this work.  Thus, we directly apply the results in \cite{he2018local}  to here. Note that  $\mu_{{\rm opt}, Q-\sigma-Uzawa}^2=\mu_{{\rm opt},Q-IBSR}$ (see Table \ref{parameters-choosing}), which means that one cycle of multigrid with $Q$-IBSR brings about the same total reduction in error as two  cycles using  $Q$-$\sigma$-Uzawa . Since  $\mu_{{\rm opt}, Q-IBSR}=\mu_{{\rm opt},Q-DR}$,  the costs of $Q$-IBSR and $Q$-DR are the same. The cost of two  sweeps of $Q$-$\sigma$-Uzawa is slightly more than one sweep of $Q$-IBSR, and  the cost of one sweep of  $Q$-DR  is a slightly more than the cost of one sweep of $Q$-IBSR. As a result, $Q$-IBSR  is the most efficient  one among these three.

\begin{table}[H]
 \caption{Optimal smoothing factors, see Definition \ref{def-opt}}
 \centering
\begin{tabular}{ c | c  c   c   c  c|}
\hline
 Old Relaxation \cite{he2018local} & DWJ  & BSR  & IBSR     & $\sigma$-Uzawa    \\
 
$\mu_{{\rm opt}}$       & $\frac{3}{5}$        &$\frac{3}{5}$    & $\frac{3}{5}$         &$\sqrt{\frac{3}{5}}$   \\
\hline
 
New Relaxation  & $Q$-DR  & $Q$-BSR  & $Q$-IBSR     & $Q$-$\sigma$-Uzawa    \\
 
$\mu_{{\rm opt}}$       & $\frac{1}{3}$        &$\frac{1}{3}$    & $\frac{1}{3}$         &$\sqrt{\frac{1}{3}}$   \\
\hline
\end{tabular}\label{parameters-choosing}
\end{table}

\section{Conclusion}
\label{sec:concl}
 
In this work, we consider multigrid methods for solving the Stokes equations discretized by the staggered finite difference method.  We use mass matrix derived from finite element methods to approximate   the inverse of the Laplacian coming from the Stokes equations. Based on this approximation, we propose  three novel  mass-based  block-structured relaxation schemes: mass-based distributive relaxation,  mass-based Braess-Sarazin relaxation, and mass-based $\sigma$-Uzawa relaxation. Then, we present a theoretical analysis of  the optimal smoothing factor of LFA for multigrid methods  for these three relaxation schemes.  We obtain fast efficient smoothing properties. Specifically, our theoretical results show that   the mass-based  distributive relaxation and Braess-Sarazin  relaxation have  optimal smoothing factor $\frac{1}{3}$,  and the mass-based Uzawa relaxation has optimal smoothing factor $\sqrt{\frac{1}{3}}$.  Numerical results   validate  our theoretical predictions, showing that the two-grid convergence factors are the same as the optimal smoothing factors.  The main advantages of these mass-based relaxation schemes are that  the mass matrix is sparse, the implementation is easy, and there is no need to compute the inverse of a matrix. Our findings here  give us a better understanding of the construction of efficient multigrid methods for the Stokes equations. It may be useful for more complex problems, such as Navier-Stokes equations, which remains a topic for  future work.
 
\section*{Acknowledgments}
The author  would like to thank Chen Greif for his careful proofreading and helpful comments
on an earlier version of this paper. 
\bibliographystyle{siam}
\bibliography{MAC_ref}
\end{document}